\documentclass{article}
\usepackage{amsmath}
\usepackage{amsthm}
\usepackage{amssymb}
\usepackage{amsfonts}
\usepackage{bbm}
\usepackage{comment}
\usepackage{array}
\usepackage{graphicx} 
\usepackage[dvipsnames]{xcolor}
\usepackage{caption}
\usepackage{mathdots}
\usepackage{subcaption}
\usepackage{makecell}
\usepackage{array}
\usepackage{mathtools}
\usepackage{tikz,ifthen}
\usetikzlibrary{math}
\usetikzlibrary{quotes}
\usetikzlibrary{positioning}

\newcolumntype{?}{!{\vrule width 2.5pt}}

\newtheorem{theorem}{Theorem}[section]

\newtheorem{lemma}[theorem]{Lemma}
\newtheorem{conjecture}{Conjecture}

\newtheorem{proposition}{Proposition}
\newtheoremstyle{bfnote}%
  {}{}
  {\itshape}{}
  {\bfseries}{.}
  { }{\thmname{#1}\thmnumber{ #2}\thmnote{ (#3)}}
\theoremstyle{bfnote}
\newtheorem{defi}{Definition}

\begin{document}

\title{An urn model for opinion propagation on networks}
\author{Andrew Melchionna}
\maketitle

\begin{abstract}
We consider a coupled P\'{o}lya's urn scheme for social dynamics on networks. Agents hold continuum-valued opinions on a two-state issue and randomly converse with their neighbors on a graph, agreeing on one of the two states. The probability of agreeing on a given state is a simple function of both of agents' opinions, with higher importance given to agents who have participated in more conversations. Opinions are then updated based on the results of the conversation. We show that this system is governed by a discrete version of the stochastic heat equation, and prove that the system reaches a consensus of opinion.
\end{abstract}

\section{Introduction}
\subsection{Statement of Problem and Result}
Let $G = (\mathcal{V},\mathcal{E})$ be a simple, connected graph, with each vertex $i \in \mathcal{V}$ representing an individual agent. In our model of opinion propagation, agents discuss an issue with their neighbors, each conversation resulting randomly in either an agreement on state $U$ or an agreement on state $V$. If two learners agree on state $U$ or $V$, both of the learners increase their propensity to prefer state $U$ or $V$, respectively, in the future. We make this precise in the following discussion.

For every vertex $i \in \mathcal{V}$ and timestep $t \in \mathbb{N} \cup \{0\},$ let the weights $(u^i_t,v^i_t) \geq 0$ represent the propensities of vertex $i$ for $U$ and $V$, respectively, at time $t$. For ease of notation, we write $(\vec{u}_t, \vec{v}_t)$, where $\vec{u}_t, \vec{v}_t \in \mathbb{R}^\mathcal{V}$ have components $(u_t^i)_{i \in \mathcal{V}}$ and $(v_t^i)_{i \in \mathcal{V}}.$ For convenience, we define the total weight of vertex $i$ and the fraction of that total weight stored in state $U$ to be
\begin{align*}
g_t^i &:= u_t^i + v_t^i\\
x_t^i &:= \frac{u_t^i}{g_t^i},
\end{align*}
respectively. We consolidate notation with $\vec{g}_t$ and $\vec{x}_t,$ similarly to the above. We enforce the initial conditions $(\vec{u}_0, \vec{v}_0)$ to be such that $u_0^i + v_0^i =: g_0^i > 0$ for all $i$, and we define $\vec{\gamma}_t$ to be a vector with $\vec{\gamma}_t^i := \frac{1}{g_t^i}$ for later convenience.

The dynamics are as follows: at every timestep $t \geq 1$ choose a random edge $e = (i,j) \in \mathcal{E}$. Increment (only) each the two $g$ values:
\begin{align*}
g_t^i &= g_{t-1}^i + 1 \\
g_t^j &= g_{t-1}^j + 1
\end{align*}
with all other $g_t^k = g_{t-1}^k$ unchanged for $k \notin \{i,j\}$.
Define:
\[
p^{e}_{t} := \frac{u_{t}^i + u_{t}^j}{g^i_{t} + g^j_{t}} = \frac{x_{t}^i g_{t}^i + x_{t}^j g_{t}^j}{g^i_{t} + g^j_{t}},
\]
as the pooled opinion of agents $i$ and $j$, and let $p^{e}_{t-1}$ give the probability of $i$ and $j$ agreeing on state $U$ at time $t$, given that edge $e$ is chosen at time $t$. If the chosen $i$ and $j$ agree on state $U,$ increment each of their $u$ values:
\begin{align*}
u_t^i &= u_{t-1}^i + 1 \\
u_t^j &= u_{t-1}^j + 1.
\end{align*}
If they agree on opinion $V$, do not alter the $u$-values. All other $u_t^k = u_{t-1}^k$ for $k \notin \{i,j\}$ remain unchanged regardless of the outcome of the conversation along edge $e$.

We show that the dynamics of the system are governed by a discrete, stochastic version of the heat equation, with an "influence matrix" $L$ driving the propagation of opinions. The influence matrix acts like the graph Laplacian, but gives higher weight to vertices which have high degree, which have more conversations on average and therefore develop strong opinions more rapidly. Similarly to the graph Laplacian, the influence matrix has right-eigenvector $\vec{1}$ (the $|\mathcal{V}|$-dimensional vector with each component equal to $1$); let $a_t$ be the coordinate of $\vec{x}_t$ corresponding to $\vec{1}$ with respect to a fixed, generalized eigenbasis of $L$ (discussed below). We will refer to $a_t$ as the \textit{consensus coordinate}.

The goal of this paper is to prove the following theorem, which states that a consensus of opinion is reached in the long-time limit.
\begin{theorem}\label{main theorem}
There exists a random scalar $0 \leq a_\infty \leq 1$ such that
\[
\mathbb{E}[\|\vec{x}_t - a_\infty \vec{1}\|^2] \xrightarrow{t \rightarrow \infty} 0
\]
\end{theorem}

\subsection{Related Work}
A similar class of frameworks for opinion propagation, called voter models, also feature randomly selected pairs of agents exchanging opinions. For example, in the Deffuant model, pairs of neighbors interact only when their opinions are within some threshold of one another, with consensus and/or polarization being driven by threshold size (\cite{r0}). Another example of a voter model is the Hegselmann-Krause model, in which an agent is randomly selected to have their opinion replaced by some determinstic function of their neighbors' opinions (\cite{r4}). The model presented in this paper could perhaps be considered a stochastic voter model (stochastic in the sense that outcomes of conversations are random). A unique property of this model, however, is the pooled-experience nature of conversations, resulting in influences between agents which are random and dynamic, but which tend towards a graph-dependent object (the influence matrix). It should also be noted that this model features continuous opinions ($x_t^i \in [0,1]$) with discrete actions (agents agree on either $U$ or $V$); different combinations of opinion and action spaces are featured throughout the literature.

This model can also be compared to the DeGroot model for learning, in which updates are made according to some constant 'trust matrix' $T$: $\vec{x}_{t+1} = T \vec{x}_{t}$ (\cite{r10}). The trust matrix can represent how much each agent trusts their neighbors as well as themself, giving a weighted average of their neighbors' beliefs and their own prior opinions. Other, similar models of opinion propagation have been studied, considering the effects of agents' self-confidence and network topology on long-term behavior (\cite{r1, r2, r3, r9}). Yet another related class of models for opinion propagation are 'probabilistic fuzzy models' which include agents' perceptions of some exogenous, albeit 'fuzzy' (the exact state is unclear) variables (\cite{r11}). We finally note that much of the literature on opinion propagation focuses on simulation-based studies, while rigorous proofs are less common.

\subsection{Outline of Paper}
The rest of the paper proceeds as follows. In Section 2, we derive the fact that the behavior of $x_t$ is governed by a discrete-time stochastic heat equation, and give some important properties of the (stochastic) Laplacian operator driving the diffusion. In Section 3, we prove convergence of the consensus coordinate of $\vec{x}_t$, and in section 4, we prove the decay of $\vec{x}_t - a_t \vec{1}$ (the \textit{disagreement component}). In Section 5, we give a proof of Theorem \ref{main theorem}, and in Section 6, we provide a conjecture that may lead to future work.

\section{Stochastic Heat Equation}
\subsection{Preliminaries}
At each timestep, an edge is randomly selected to host a 'conversation' between its two vertices. The following heuristic is equivalent and useful: let all edges have conversations, and uniformly at random select one edge to actually contribute to the dynamics.

Let $\Omega_t = \{\omega_t^{e}\}_{e \in \mathcal{E}}\in \{0,1\}^{\mathcal{E}}$ be the results of all conversations occuring at timestep $t$, with $\omega_t^e = 1$ if opinion $U$ is agreed on, and $0$ otherwise. Similarly, let $\psi_t \in \mathcal{E}$ be the edge chosen at time $t$, and let $S_t^e = \mathbbm{1}_{\{\psi_t = e\}}$. Define the filtrations
\begin{align*}
\mathcal{H}_1 \subset \mathcal{H}_2 \subset \mathcal{H}_3 \subset ...\\
\mathcal{G}_1 \subset \mathcal{G}_2 \subset \mathcal{G}_3 \subset ...\\
\mathcal{F}_1 \subset \mathcal{F}_2 \subset \mathcal{F}_3 \subset ...,
\end{align*}

where $\mathcal{H}_t = \sigma(\Omega_t,\Omega_{t-1},...,\Omega_1)$ corresponds to the information received up to and immediately after discussions in the $t^{\text{th}}$ round, $\mathcal{G}_t = \sigma(\psi_t, \psi_{t-1},...,\psi_1)$ corresponds to the information received given all of the chosen edges up to and including time $t$, and let $\mathcal{F}_t = \sigma(\mathcal{H}_t, \mathcal{G}_t,...,\mathcal{H}_1, \mathcal{G}_1)$. Note that $\Omega_t \in m(\mathcal{H}_t)$, but $\Omega_{t+1} \notin m(\mathcal{H}_{t})$. Since $\psi_t$ does not care about the previous edges chosen or the results of any concurrent or previous conversations, we let $\sigma(\psi_t)$ be independent of $\sigma(\mathcal{F}_{t-1} \cup \mathcal{H}_t)$. Furthermore, for $e \neq f,$ we let $\omega_t^e$ and $\omega_t^f$ be conditionally independent given $\mathcal{F}_{t-1}$ (they are not fully independent, since they are both affected by the history of conversations over the network). Define the full sample space and sigma-algebra to be
\begin{align*}
\Omega \times \Psi &:= \big(\{U,V\}^\mathcal{E}\big)^\mathbb{N} \times \mathcal{E}^\mathbb{N} \\
\mathcal{F} &:=\sigma \big( \cup_{t \in \mathbb{N}} \mathcal{F}_t\big).
\end{align*}
Using the notation established above, we have the following update rule for $g$ and $u$:
\begin{align*}
g_{t+1}^i - g_t^i &:= \sum_{e \rightarrowtail i} S_{t+1}^e\\
u_{t+1}^i-u_t^i &:= \sum_{e \rightarrowtail i} S_{t+1}^e \omega^e_{t+1},
\end{align*}
where $e\rightarrowtail i$ means that edge $e$ is incident to vertex $i$, and we are summing over all such edges.

Decompose $\omega_t^e$ into a $\mathcal{F}_{t-1}$-measurable random variable and a mean-$0$ $\sigma(\mathcal{F}_{t-1}, \mathcal{H}_{t})$-measurable fluctuation:
\[
\omega_t^{e} = p_{t-1}^e + \tilde{w}_t^e,
\]
so that
\[
\tilde{w}_{t}^{e} = 
\begin{cases}
1 -  p_{t-1}^{e} & \text{with probability } p_{t-1}^{e}\\
-p_{t-1}^e & \text{with probability } 1 - p_{t-1}^{e}.
\end{cases}
\]

Let $w_t^i := \sum_{e \rightarrowtail i} \tilde{w}_t^e S_t^e$. From here forward, we will use the notation $\mathbb{E}_t[Z] := \mathbb{E}[Z|\mathcal{F}_t]$ to represent conditional expectation with respect to the sigma-algebra $\mathcal{F}_t$. Note that 
\begin{align*}
\mathbb{E}_{t-1}[w_t^i] &= \sum_{e \rightarrowtail i} \mathbb{E}[ \tilde{w}_t^e S_{t}^e|\mathcal{F}_{t-1} ] \\
&= \sum_{e \rightarrowtail i} \mathbb{E}\Big[  \mathbb{E}[ \tilde{w}_t^e S_{t}^e|\sigma \big(\mathcal{F}_{t-1} \cup \mathcal{H}_t \big) ]|\mathcal{F}_{t-1}\Big ]\\
&=\sum_{e \rightarrowtail i} \mathbb{E}\Big[ \tilde{w}_t^e \mathbb{E}[  S_{t}^e|\sigma \big(\mathcal{F}_{t-1} \cup  \mathcal{H}_t\big) ]|\mathcal{F}_{t-1}\Big ]\\
&= \frac{1}{|\mathcal{E}|}\sum_{e \rightarrowtail i} \mathbb{E}\Big[ \tilde{w}_t^e|\mathcal{F}_{t-1}\Big ]\\
&= 0,
\end{align*}
where in the second, third, and fourth equalities we've used the tower property, 'taken out what was known', and used that $S^e_t \perp \sigma(\mathcal{F}_{t-1} \cup \mathcal{H}_t),$ respectively.

For later convenience, we present here a consolidated list of definitions of important quantities, and the earliest sigma-algebra $\mathcal{F}$ with respect to which they are measurable:
\begin{defi}[Important Quantities] \leavevmode \\
\begin{itemize}
\item \textbf{Total weight}: $g_t^i \in m\mathcal{F}_t$, $\gamma_t^i = \frac{1}{g_t^i}$
\item \textbf{Weight on opinion $U$}: $u_t^i \in m\mathcal{F}_t$
\item \textbf{Proportion of weight on $U$}: $x_t^i = \frac{u_t^i}{g_t^i}\in m \mathcal{F}_t$
\item \textbf{Initial conditions}: $u_0^i,g_0^i > 0$
\item \textbf{Mutual weight on $U$}: $p_t^{e} = \frac{u_t^i+u_t^j}{g_t^i + g_t^j} \in m\mathcal{F}_t,$ where $e = (i,j)$
\item \textbf{Mean-0 fluctuation of conversation result}: $\tilde{w}_t^{e} \in m\mathcal{F}_t$ 
\item \textbf{Result of conversation}: $\omega_t^{e} = p_{t-1}^{e} + \tilde{w}_t^{e} \in m\mathcal{F}_t$
\item \textbf{Edge to play}: $\psi_t \in m \mathcal{F}_t$, $S_t^e = \mathbbm{1}_{\{\psi_t = e\}} \in m\mathcal{F}_t$
\end{itemize}

\end{defi}

We now define a Hadamard (elementwise) product between a vector and a matrix. Unless otherwise noted, the symbol $\|\cdot \|$ will refer to the Euclidean norm for vectors, and the operator norm between Euclidean vector spaces for matrices. We carry this convention through the end of the paper.
\begin{defi}[Hadamard Product]
The left-Hadamard product between an $m$-dimensional row vector $b$ and $(m \times n)$ matrix $A$ is a $(m \times n)$ matrix with entries given as follows:
\[
(b \circ_L A)^{ij} := b^i A^{ij}.
\]
Similarly, the right-Hadamard product between an $n$-dimensional column vector and  $(m \times n)$ matrix $A$ is an $(m \times n)$ matrix with entries as follows:
\[
(A \circ_R b)^{ij} = A^{ij} b^j.
\]
\end{defi}
We will omit subscripts $L$ and $R$ when it is clear from the context what is meant. It can readily be shown that the Euclidean norms are sub-multiplicative with respect to the right-Hadamard product. For an $m\times n$ matrix $A$ and an $n$-column vector $b$,
\[
\|A \circ b\| \leq \|A\| \|b\|.
\]
and similarly for left-products. Another important property of Hadamard multiplication is its associativity with matrix multiplication. For an $(m \times n)$ matrix $A_1$, an $n \times p$ matrix $A_2,$ and an $n-$column vector $b$,
\[
A_1 (b^T \circ_L A_2) = (A_1 \circ_R b) A_2.
\]

\subsection{Deriving the Stochastic Heat Equation}
Fix an arbitrary vertex $i \in V$. We now consider the quantity $u_{t+1}^i-u_t^i$, which represents the increase in the propensity of vertex $i$ to play move $u$ after timestep $t+1$.
\[
u_{t+1}^i - u_t^i =  \sum_{e \rightarrowtail i} S^e_{t+1} \omega_{t+1}^{e}= \sum_{e \rightarrowtail i}S^e_{t+1}   \Big( p_{t}^{e} + \tilde{w}_{t+1}^{e} \Big)
\]
We use the equation above to write down the change in $x^i$ between timesteps $t$ and $t+1$:
\begin{align*}
x_{t+1}^i-x_t^i &= \frac{u_{t+1}^i}{g_{t+1}^i} - \frac{u_{t}^i}{g_{t}^i} = \frac{1}{g_{t+1}^i} \Big( u^i_{t+1}-u^i_t - \frac{g^i_{t+1}-g^i_t}{g^i_t}u^i_t \Big) \\
&=  \frac{1}{g_{t+1}^i} \Big(  \sum_{e \rightarrowtail i}S^e_{t+1}   \Big( p_{t}^{e} + \tilde{w}_{t+1}^{e} \Big) - (g^i_{t+1}-g^i_t) x^i_t\Big) \\
&=  \frac{ w_{t+1}^i}{g_{t+1}^i} +\frac{1}{g_{t+1}^i} \Big( \Big[\sum_{j\sim i} S_{t+1}^{ij}  \frac{x_{t}^i g_t^i + g_t^j x_{t}^j}{g_t^i + g_t^j}\Big] -  (g^i_{t+1}-g^i_t) x^i_t\Big) \\
&=  \frac{1}{g_{t+1}^i} \big(  w_{t+1}^i + \big(L_t \vec{x}_t\big)^i \big),
\end{align*}
where $L_t$ is defined as follows:
\begin{defi}[The Diffusion Matrix]

The diffusion matrix $L_t\in m\mathcal{F}_{t+1}$ is a $|\mathcal{V}| \times |\mathcal{V}|$ matrix with entries:
\begin{align*}
L_{t}^{ij} &= \begin{cases}
S_{t+1}^{ij}\frac{g_t^j}{g_t^i + g_t^j} & i \neq j, i \sim j \\
-\sum_{j \sim i}S_{t+1}^{ij}\frac{g_t^j}{g_t^i + g_t^j} & i = j \\
0 & \text{else}
\end{cases}
\end{align*}
We also define $\Lambda_t \in m \mathcal{F}_{t+1}$ to be a $|\mathcal{V}| \times |\mathcal{V}|$ matrix as follows:
\begin{align*}
\Lambda_t &= I + \gamma_{t+1}^T \circ L_t.
\end{align*}
\end{defi}
Note that $L_t$ will have exactly four non-zero entries, and takes the following form:
\[
\begin{bmatrix}
0&\cdots&\cdots&\cdots&\cdots&\cdots&0 \\
\vdots&\ddots&\cdots&\cdots&\cdots&\cdots&\vdots \\
\vdots&\cdots&-a&\cdots&a&\cdots&\vdots \\
\vdots&\cdots&\cdots&\ddots&\cdots&\cdots&\vdots \\
\vdots&\cdots&b&\cdots&-b&\cdots&\vdots \\
\vdots&\cdots&\cdots&\cdots&\cdots&\ddots&\vdots \\
0&\cdots&\cdots&\cdots&\cdots&\cdots&0
\end{bmatrix}
\]
for some $a,b > 0$.

Although $L_t$ is sparse, its expectation given the previous timestep, $\mathbb{E}_t[L_t]$, is worthy of mention. It represents the aggregate effects after many rounds of conversations:
\begin{align*}
\mathbb{E}_t[L_{t}]^{ij} &= \begin{cases}
\frac{1}{\mathcal{E}} \frac{g_t^j}{g_t^i + g_t^j} & i \neq j, i \sim j \\
-\frac{1}{\mathcal{E}} \sum_{j \sim i}\frac{g_t^j}{g_t^i + g_t^j} & i = j \\
0 & \text{else}.
\end{cases}
\end{align*}

We also note that each $g_t^i - g_0^i$ is a binomial random variable with mean equal to $t \frac{d_i}{E}$, where $d_i$ is the degree of vertex $i$. We thus expect the leading order terms of $\mathbb{E}_t[L_t]$ to look like $\frac{1}{\mathcal{E}}$ times the following \textit{influence matrix}, a graph dependent constant, defined below. 
BREAK
\begin{defi}[Influence Matrix]
The influence matrix $L$ is a $|\mathcal{V}| \times |\mathcal{V}|$ matrix with entries:
\begin{align*}
L^{ij} &= \begin{cases}
\frac{d_j/d_i}{d_i+d_j } & i \neq j, i \sim j \\
-\sum_{j \sim i} \frac{d_j/d_i}{d_i+d_j } & i = j \\
0 & \text{else}.
\end{cases}\\
\end{align*}
We also define $A_t$ to be a $|\mathcal{V}| \times |\mathcal{V}|$ matrix as follows:
\begin{align*}
A_t &= I + \frac{1}{t}L
\end{align*}
\end{defi}
The influence matrix corresponds to the graph Laplacian matrix for the weighted, directed graph $\mathcal{I}(G) := \big(\mathcal{V},\mathcal{I}(\mathcal{E})\big)$, where $(i,j) \in \mathcal{E} \iff (i,j),(j,i) \in \mathcal{I}(\mathcal{E})$, and the edge weight from $j$ to $i$ is defined to be $L^{ij}$ (see Figure 1). Note that edge weights from $j$ to $i$ are high when $d_j$ is large relative to $d_i$. We think of $j$ as having more 'influence' than $i$ in this case.

\begin{figure}
\centering
\begin{subfigure}{0.7\linewidth}
\begin{center}
\begin{tikzpicture}[main node/.style={circle,fill=gray!20,draw,minimum size=.8cm,inner sep=0pt}]
     \node[main node] (1) {$1$};
    \node[main node] (2) [right = 1.5cm of 1]  {$2$};
    \node[main node] (3) [right = 1.5cm of 2] {$3$};

    \path[draw,very thick]
    (1) edge node {} (2)
    (2) edge node {} (3);
\end{tikzpicture}
\end{center}
\caption{$G$}
\end{subfigure}
\vspace{.5cm}
\vfill
\begin{subfigure}{0.5\linewidth}
\begin{center}
\begin{tikzpicture}[->,main node/.style={circle,fill=gray!20,draw,minimum size=.8cm,inner sep=0pt}]
    \node[main node] (1) {$1$};
    \node[main node] (2) [right = 1.5cm of 1]  {$2$};
    \node[main node] (3) [right = 1.5cm of 2] {$3$};

    \path[draw,very thick]
   (2) edge[bend left] node [below]{$\frac{2}{3}$} (1)
   (1) edge[bend left] node [above]{$\frac{1}{6}$} (2)
   (3) edge[bend right] node [above]{$\frac{1}{6}$} (2)
     (2) edge[bend right] node [below]{$\frac{2}{3}$} (3);
\end{tikzpicture}
\end{center}
\caption{$\mathcal{I}(G)$}
\end{subfigure}
\caption{A graph $G$ and the directed, weighted graph $\mathcal{I}(G)$ associated to $G$}
\end{figure}
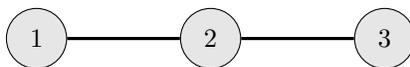
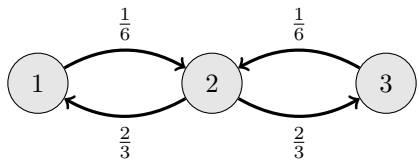

With these definitions in place, we present the Stochastic Heat Equation (abbreviated SHE), derived above:
\begin{proposition}[Stochastic Heat Equation]{\label{SHE}}
We present the differential form of the Stochastic Heat Equation (SHE):
\[
\partial_t x_t =  \gamma_t^T \circ (L_{t-1}\vec{x}_{t-1} + W_t),
\]
and its solution:
\[
\vec{x}_t = \sum_{j=0}^t [\Pi_{k=j}^{t-1} \Lambda_k ] (\gamma_j^T \circ W_j)
\]
where $\partial_t \vec{x}_t := \vec{x}_t - \vec{x}_{t-1}$ and $\vec{W}_0 := \vec{u}_0.$
\end{proposition}
Throughout the paper, we will use the convention that $\Pi$- products of matrices have older matrices to the right, for example:
\[
\Pi_{k=1}^t \Lambda_k = \Lambda_k \Lambda_{k-1} \cdot \cdot \cdot \Lambda_2 \Lambda_1.
\]

As intuition may suggest, the steady-state solution to the above heat equation is consensus: all $x_t^i$ will converge to the same (random) constant. At the heart of this idea is the Perron-Frobenius Theorem, which says that the eigenvector $\vec{1}$ of $I + L$ which represents consensus has strictly dominant eigenvalue $1$. We first state the Perron-Frobenius theorem for nonnegative matrices (Lemma \ref{pfgen}), along with another necessary technical ingredient (Lemma \ref{primitive}).

\begin{lemma}{\cite{r13, r14}}{\label{pfgen}}
Let $M$ be a square, nonnegative, irreducible, primitive matrix (i.e., there exists $k >0$ such that $M^k > 0$ elementwise) with spectral radius $\rho$. Then the following hold:
\begin{itemize}
\item $\rho$ is an algebraically simple eigenvalue of $M$, and the corresponding normalized eigenvector $\vec{v}$ is unique and positive
\item Any nonnegative eigenvecor of $M$ is a multiple of $\vec{v}$
\item All other eigenvalues of $M$ have absolute value strictly smaller than $\rho$
\end{itemize}
\end{lemma}

\begin{lemma}{\cite{r15}}{\label{primitive}}
Let $M$ be an $n\times n$ matrix, and define $\Gamma(M)$ to be a digraph with vertex set $\mathcal{V} = \{1,...,n\}$ and directed edge set $\mathcal{E} = \{(i,j) : M_{ij} \neq 0 \}$. If $\Gamma(M)$ is strongly connected, and every vertex $i$ of $\Gamma(M)$ has a self-loop, then $M$ is primitive.
\end{lemma}

Having stated the above two ingredients, we now apply Perron-Frobenius to our system in the lemma below.

\begin{lemma}{\label{pfspec}}
For all $k \geq 1$, $1$ is a simple eigenvalue of $A_k := I + \frac{1}{k} L$. Furthermore, there exists $ 0<\lambda<1$ such that for all $k \geq 1,$ and for all eigenvalues $\lambda^{(k)} \neq 1$ of $A_k$:
\[
|\lambda^{(k)}| \leq 1- \frac{\lambda}{k}.
\]
\end{lemma}
\begin{proof}
First notice that, for each row $i$ and for all times $t$, $\sum_j L^{ij} = \sum_j  (\vec{\gamma}_{t+1}^T \circ L_t)^{ij}=0 .$ From this it immediately follows that $0$ is an eigenvalue of both matrices with corresponding right-eigenvector
\[
\vec{1} :=  \begin{bmatrix}1 \\ \cdot \\ \cdot \\ \cdot \\ 1 \end{bmatrix},
\]
and thus that $\vec{1}$ is a right-eigenvector of $A_k$ with eigenvalue $1$. Next, label the eigenvalues $\mu_i$ of $L$ such that $\mu_1 = 0$. Notice that for any $k \geq 1,$ the eigenvalues $\lambda_i^{(k)}$ of $A_k$ are given by $\lambda_i^{(k)} = 1 + \frac{\mu_i}{k}$, numbered such that $\lambda_1^{(k)} = 1$ for all $k$. It remains to show that $1$ is a simple eigenvalue of $A_k$ and the bound given above. 

We invoke the Perron-Frobenius theorem for irreducible non-negative matrices on $A_k := I + \frac{1}{k} L$. $A_k$ is nonnegative since it is clear that all off-diagonal elements are nonnegative, and for all $i$,
\[
L^{ii} = -\sum_{j \sim i} \frac{\frac{\deg(j)}{\deg(i)}}{\deg(i) + \deg(j)} >  -\sum_{j \sim i} \frac{1}{\deg(i)} = -1.
\]
This gives that $A_{k}^{ii} > 0$ and thus that $A_k$ is nonnegative. 

In order to show that $A_k$ is irreducible, we consider its associated weighted digraph $\Gamma(A_k)$, which has vertex set $V,$ a complete edge set $V \times V,$ and weights $W: V \times V \rightarrow \mathbb{R}_{\geq 0}.$ By the definition of $L$, we have that for all $i \sim j$ in the original graph,   there are edges with non-zero weights flowing from $i$ to $j$ and from $j$ to $i$. Since the original graph $V$ is connected, this implies that the weighted digraph associated to $A_k$ is strongly connected, giving that $A_k$ is irreducible. Also note that since the diagonal elements of $A_k$ are all strictly positive, each vertex in $\Gamma(A_k)$ has a self-loop, and thus $A_k$ is primitive by Lemma \ref{primitive}. Thus $A_k$ satisfies the assumptions of the Lemma \ref{pfgen}. Since the eigenvector $\hat{1}$ has components which are all positive, Perron-Frobenius gives that associated eigenvalue $1$ of $A_k$ is simple, that the spectral radius of $A_k$ is $1$, and that all other eigenvalues of $A_k$ have modulus strictly less than $1$.

Let $\lambda$ represent the spectral gap of $A_1$ (unless the spectral gap is $1$, in which case we can arbitrarily set $\lambda = \frac{1}{2}$):
\[
\lambda := \begin{cases}
1-\max_{i>1}|\lambda^{(1)}_i| & \text{if } \max_{i>1}|\lambda^{(1)}_i|  \neq 0 \\
\frac{1}{2} & \text{else}
\end{cases}
\]
This gives that, for all $k\geq 1$ and $i > 1,$
\[
|\lambda_i^{(k)}| = |1 + \frac{\mu_i}{k}| = \frac{1}{k} |\mu_i + k| \leq \frac{1}{k}|\lambda_i^{(1)}| + \frac{k-1}{k} \leq \frac{k-\lambda}{k} = 1-\frac{\lambda}{k}.
\]
\end{proof}

From here forward, we let $\lambda$ represent the number guaranteed by the above lemma. The next lemma shows that $L$ is similar to a symmetric matrix and hence is diagonalizable, which simplifies the long-time analysis involving products of $L$.
\begin{lemma}{\label{diag}}
$L$ is diagonalizable, and can be written $L = PDP^{-1},$ where the first column of $P$ is $\vec{1},$ and $D_{11} = 0.$
\end{lemma}
\begin{proof}
Let $E$ be the diagonal matrix with diagonal elements equal to the degree of each vertex:
\[
E^{ij} = \begin{cases}
d_i & i = j\\
0 & i \neq j
\end{cases}.
\]
Note that $E$ has strictly positive entries on the diagonal and is therefore invertible with
\[
(E^{-1})^{ij} = \begin{cases}
\frac{1}{d_i} & i = j\\
0 & i \neq j
\end{cases}.
\]
Then note that $ELE^{-1}$ is symmetric, because
\begin{align*}
\big(ELE^{-1}\big)^{ij} &= \sum_{k,\ell} E^{ik} L^{k \ell} (E^{-1})^{\ell j} \\
=\frac{d_i}{d_j} L^{ij}.
\end{align*}
Now, by definition of $L$: if $i \nsim j,$ then $(ELE^{-1})^{ij} =(ELE^{-1})^{ji} = 0$, and if $i \sim j,$ then $(ELE^{-1})^{ij} = \frac{1}{d_i + d_j} = (ELE^{-1})^{ji}.$ Thus $(ELE^{-1})$ is symmetric and therefore diagonalizable. Since $L$ is similar to a diagonalizable matrix, it is itself diagonalizable.
\end{proof}
From here forward, we fix $P$ and $D$ as given in Lemma \ref{diag}. The above two lemmas make a powerful combination, in the following sense. Note that the solution to the SHE (Proposition \ref{SHE}) involves a product of the $\Lambda$ matrices: $\Pi_{k = j}^{t-1} \Lambda_k.$ In the discussion below, we show that this large product can be approximated by the following product of constant matrices: $\Pi_{k = j}^{t-1} A_k$, which can in turn is similar to a product of diagonal matrices: $\Pi_{k = j}^{t-1} D_k$, where $D_k = P^{-1}A_k P$. Now, while the first entry of each of the $D_k$ is $1$ (corresponding to consensus), the other entries are bounded by $1- \frac{\lambda}{k}$ (due to Lemma \ref{pfspec}). The last ingredient of this section is an application of the theory of gamma functions, due to Gautschi, which shows that while these eigenvalues approach $1$ from below as $k \rightarrow \infty$, the approach is slow enough for their product to approach 0.

\begin{lemma}{\cite{r16}}{\label{gautschi}}
For $0<s<1$:
\[
x^{1-s} < \frac{\Gamma(x+1)}{\Gamma(x+s)} < (x+1)^{1-s}
\]
\end{lemma}

\begin{lemma}{\label{gautschiapp}}
For all $1 \leq j \leq t$ and for $0<\lambda < 1,$
\[
 \Big(\frac{j-1}{t+1}\Big)^\lambda \leq \Pi_{k=j}^t (1 - \frac{\lambda}{k}) \leq \Big(\frac{j}{t}\Big)^\lambda
\]
\end{lemma}
\begin{proof}
We write 
\[
\Pi_{k=j}^t (1 - \frac{\lambda}{k}) = \frac{\Pi_{k=j}^t (k-\lambda)}{\Pi_{k=j}^t (k)} = \frac{\Gamma(j)}{\Gamma(j-\lambda)} \frac{\Gamma(t+1-\lambda)}{\Gamma(t+1)},
\]
and apply Gautschi's inequality (Lemma \ref{gautschi}).
\end{proof}

\section{Convergence of the Consensus Coordinate}

Let $\vec{p}$ be the first row of $P^{-1},$ i.e. the left-eigenvector of $L$ with eigenvalue $0$, and let $a_t:= p \cdot x_t$ be the coordinate corresponding to $\vec{p}$ in the eigenbasis expansion of $L$ (where the eigenbasis is given by the columns of $P$). The goal of this section is to show the following lemma.

\begin{lemma}{\label{consensusmain}}
There exists a random constant $a_\infty$ such that $a_t \rightarrow a_\infty$ in $\mathcal{L}^2.$ 
\end{lemma}

We decompose $a_t$ as follows:
\begin{align*}
a_t &= a_0 + \sum_{j = 0}^{t-1} (a_{j+1}-a_j)\\
&= a_0 + \vec{p} \cdot \sum_{j = 0}^{t-1} (\vec{x}_{j+1}-\vec{x}_j)\\
&= a_0 + \vec{p} \cdot \sum_{j = 0}^{t-1} \vec{\gamma}^T_{j+1}\circ L_j \vec{x}_j + \gamma^T_{j+1}\circ \vec{w}_{j+1} \\
&=a_0 + \vec{p} \cdot \sum_{j = 0}^{t-1} (\vec{\gamma}^T_{j+1}\circ L_j  -  \frac{1}{j+1}L)\vec{x}_j + \frac{1}{j+1}L \vec{x}_j + \vec{\gamma}^T_{j+1}\circ \vec{w}_{j+1}  \\
&=a_0 + \vec{p} \cdot \sum_{j = 0}^{t-1} (\vec{\gamma}^T_{j+1}\circ L_j  -  \frac{1}{j+1}L)\vec{x}_j  + \vec{\gamma}^T_{j+1}\circ \vec{w}_{j+1}  \\
\end{align*}
where we've used the SHE update in the third line, and the fact that $\vec{p}$ is a left 0-eigenvector of $L$ in the fifth.

Now, there are two main differences between the dampened diffusion matrix $\vec{\gamma}^T_{j+1}\circ L_j$ and the dampened influence matrix $\frac{1}{j+1} L.$ The first is that the diffusion matrix only involves a random edge, while the influence matrix considers all edges. The second is that the $g_t$ are random functions of the $\psi$ variables, while $L$ is a constant. We separate out these two differences by adding and subtracting $\mathbb{E}_j[\vec{\gamma}^T_{j+1}\circ L_j]$:
\begin{align*}
a_t &=a_0 + m_t + s_t,
\end{align*}
where
\begin{align*}
m_t &:= \vec{p} \cdot \sum_{j = 0}^{t-1} ( \gamma^T_{j+1}\circ L_j -\mathbb{E}_j[ \gamma^T_{j+1}\circ L_j]) \vec{x}_j + \vec{\gamma}^T_{j+1}\circ \vec{w}_{j+1}\\
s_t &:= \vec{p} \cdot \sum_{j = 0}^{t-1}\Delta_j \vec{x}_j\\
\Delta_j &:=  \mathbb{E}_j[\vec{\gamma}^T_{j+1}\circ L_j] - \frac{1}{j+1} L
\end{align*}
We consider each of $s_t$ (which stands for 'small') and $m_t$ (which stands for martingale) separately; in order to show Lemma \ref{consensusmain}, it suffices to show that each of $s_t$ and $m_t$ converge in $\mathcal{L}^2$. While $s_t$ is nonzero due to the randomness of $g_t$, we show that each term is small in expectation and therefore that the sum is convergent, while $m_t$ is shown to be a martingale, on which we will invoke the martingale convergence theorem.

Before proceeding, we state a useful lemma which allows us to rigorously pass from sums to integrals:

\begin{lemma}{\label{sumint}}
Let $f(k)$ be nonnegative on $[t_1,t_2]$, non-decreasing on $[t_1,x]$ and non-increasing on $[x,t_2]$ for some $x \in [t_1,t_2].$ Then
\[
\sum_{k=t_1}^{t_2} f(k) \leq \int_{t_1}^{t_2} f(k) dk +  2 f(x)
\]
\end{lemma}
\begin{proof}
In the below, we take sums with lower endpoint strictly greater than upper endpoint to be $0$. We have:
\begin{align*}
\sum_{k=t_1}^{t_2} f(k) &\leq 2 f(x) + \sum_{k=t_1}^{\lfloor x \rfloor -1 } f(k) + \sum_{k= \lceil x \rceil + 1}^{t_2} f(k) \\
& \leq  \int_{t_1}^{\lfloor x \rfloor} f(k) + \int_{ \lceil x \rceil}^{t_2} f(k) + 2 f(x) \\
& \leq \int_{t_1}^{t_2} f(k) dk + 2 f(x).
\end{align*}
\end{proof}

\subsection{$s_t:$ Fluctuations of $\vec{g}_t$}
Recall that
\[
\mathbb{E}_t [\vec{\gamma}_{t+1}^T \circ L_t ]^{ij}= 
\begin{cases}
\frac{1}{E(g^i_t + 1)}\frac{g^j_t}{g^i_t + g^j_t} &  i \sim j \\
-\sum_{k \sim i} \frac{1}{E(g^i_t + 1)}\frac{g^k_t}{g^i_t + g^k_t} & i=j\\
0 & i \nsim j.
\end{cases}
\]
Now, the random variable $g_t^i$ is equal to $g_0^i$ plus a binomial random variable resulting from $t$ trials with probability $\frac{d_i}{E}$ of success for each trial. Thus we expect each $g_t^i$ to grow like $\frac{d_i}{E}t$, with standard deviation proportional to $\sqrt{t}$. This gives the heuristic that $\mathbb{E}[\|\Delta_j\|] = O(t^{-\frac{3}{2}})$. This idea is supported by the following concentration inequality for the binomial random variable, which can be used to show that the probability of $g_t^i - g_0^i$ deviating from its mean by $t^{\frac{1}{2} + \epsilon}$ is exponentially small in $t$. 

\begin{lemma}{\cite{r17}}{\label{AS}}
Let $B\sim \text{Bin}(n,p)$ be a binomial random variable, and let $a>0$. Then
\[
\mathbb{P}(|B-np| > a) < 2 \exp(-\frac{2a^2}{n}).
\]
\end{lemma}
The above statement serves as the main tool for showing that $\Delta_j$ is indeed small. In particular, we use prove the following Lemma, which will be used to show that $s_t$ converges in $\mathcal{L}^2.$ From here forward, we use the notation $f(t) \lesssim g(t)$ to mean that there exists a constant $c$, independent of $t$, such that $f(t) \leq c g(t).$
\begin{lemma}{\label{delta}}
For sufficiently large $s, t$:
\begin{align*}
\mathbb{E}[\|\Delta_s\| \|\Delta_t\|] &\lesssim  \frac{1}{(st)^{\frac{5}{4}}}+ \frac{s}{t^{\frac{5}{4}}} \exp(- \frac{2}{|\mathcal{E}|^2}s^{1/2}) +st\exp(- \frac{2}{|\mathcal{E}|^2}t^{1/2})
\end{align*}
\end{lemma}
\begin{proof}
Define
\begin{align*}
C_t &:= \{\bigl| |\mathcal{E}| \frac{g_t^i- g_0^i}{t} - d_i \bigr|\leq \frac{1}{t^{\frac{1}{4}}} \, : \, \forall i \in \mathcal{V}\}\\
\delta_t^i &:= |\mathcal{E}|\frac{g_t^i- g_0^i}{t} + |\mathcal{E}| \frac{g_0^i}{t} -  d_i,
\end{align*}
and note that $|\mathcal{E}| \frac{g_0^i}{t} - d_i \leq \delta_t^i \leq |\mathcal{E}| (g_0^i + 1)- d_i$ almost surely. Note also that using $a = \frac{t^{3/4}}{|\mathcal{E}|}$ in Lemma \ref{AS} produces
\[
\mathbb{P}(C^c_t) < 2|\mathcal{V}|\exp(- \frac{2}{|\mathcal{E}|^2}t^{\frac{1}{2}}). 
\]
Now, for fixed $i \sim j$, we have that:
\begin{align*}
(t+1) \Delta_t^{ij} &= \frac{t+1}{|\mathcal{E}|(g^i_{t}+1)} \frac{g_t^j}{g_t^i + g_t^j} - \frac{\frac{d_j}{d_i}}{d_i +d_j }\\
& = \frac{1+\frac{1}{t}}{(\delta^i_{t} + d_i+\frac{|\mathcal{E}|}{t})} \frac{\delta_t^j + d_j}{\delta_t^j + d_j + \delta_t^i + d_i} - \frac{d_j}{d_i(d_i + d_j) }\\
&= \frac{(1+\frac{1}{t})(d_i + d_j)(\delta_t^j + d_j)d_i-d_j (\delta^i_{t} + d_i+\frac{|\mathcal{E}|}{t})(\delta_t^j + d_j + \delta_t^i + d_i) }{d_i(d_i + d_j)(\delta^i_{t} + d_i+\frac{|\mathcal{E}|}{t})(\delta_t^j + d_j + \delta_t^i + d_i)}\\
\end{align*}
Almost surely:
\begin{align*}
(t+1) |\Delta_t^{ij}|&\leq \frac{c_1 |\delta_t^i| + c_2 |\delta_t^j| + c_3 \frac{1}{t}}{(d_i + \delta_{t}^i)(d_j + \delta_{t}^j)}
\end{align*}
for some $t$-independent constants $c_1,c_2,c_3.$
Further, on $C_t,$ $|\delta_t^i| \lesssim \frac{1}{t^{1/4}}$ for all $i$. So, on $C_t$ for sufficiently large $t$,
\[
|\Delta_t^{ij}|\lesssim \frac{1}{t^{5/4}}.
\]
It's also easy to see that, almost surely (in particular, on $C_t^c$),
\[
|\Delta_t^{ij}| \lesssim t
\]
Further, since $\Delta_t$ is row-stochastic for all $t$, we can drop the requirement that $i \sim j$ for the above two inequalities on $|\Delta^{ij}_t|$ (perhaps at the cost of a larger constant). 

Now, for $s \neq t,$
\begin{align*}
 \mathbb{E} [\|\Delta_s\|_\text{max} \|\Delta_t\|_\text{max}] & =  \mathbb{E} [\max_{i,j,k,\ell}|\Delta^{ij}_s \Delta^{k \ell}_t||C_s \cap C_t] \mathbb{P}(C_s \cap C_t)\\& +   \mathbb{E} [\max_{i,j,k,\ell}|\Delta^{ij}_s \Delta^{k \ell}_t||C_s^c \cap C_t ] \mathbb{P}(C_s^c \cap C_t) +   \mathbb{E} [\max_{i,j,k,\ell}|\Delta^{ij}_s \Delta^{k \ell}_t||C_t^c] \mathbb{P}(C_t^c) \\
& \lesssim \frac{1}{(st)^{\frac{5}{4}}}+ \frac{s}{t^{\frac{5}{4}}} \exp(- \frac{2}{\mathcal{E}^2}s^{1/2}) + st\exp(- \frac{2}{\mathcal{E}^2}t^{1/2}),
\end{align*}
where $\|A\|_{\max} := \max_{i,j} |A^{ij}|$. Finally, note that
$\|\Delta_t\| \lesssim \|\Delta_t\|_\text{max}$. This gives the desired result for $s \neq t.$

When $s = t,$ we have:
\begin{align*}
 \mathbb{E} [\max_{i,j}|\Delta^{ij}_t|^2] & \leq  \mathbb{E} [\max_{i,j}|\Delta^{ij}_t|^2|| C_t] \mathbb{P}(C_t) +   \mathbb{E} [\max_{i,j}|\Delta^{ij}_t|^2|C_t^c ] \mathbb{P}(C_t^c) \\
&\lesssim \frac{1}{t^{\frac{5}{2}}} + t^2 \exp(-\frac{2}{|\mathcal{E}|}t^{1/2}),
\end{align*}
and again we use that $\|\Delta_t\| \lesssim \|\Delta_t\|_\text{max}$.
\end{proof}
This allows us to prove the desired convergence of $s_t$.
\begin{lemma}{\label{sconverges}}
$s_t$ converges in $\mathcal{L}^2$.
\end{lemma}
\begin{proof}
It suffices to show Cauchy in $\mathcal{L}^2,$ i.e. that for any $\epsilon > 0,$ there exists $T$ such that for all $t_1,t_2 > T,$ $\mathbb{E}[(s_{t_1}-s_{t_2})^2] \leq \epsilon$. Note that
\[
(s_{t_2} - s_{t_1})^2 \lesssim \sum_{j,k = t_1}^{t_2-1}  \| \Delta_j\|_2 \|  \Delta_k \|_2 \lesssim \sum_{j=t_1}^{t_2-1} \sum_{k = t_1}^j  \| \Delta_j\|_2 \|  \Delta_k \|_2 
\]
Taking expectations and using the lemma, 
\[
\mathbb{E} [(s_{t_2} - s_{t_1} )^2] \lesssim \sum_{j=t_1}^{t_2-1} \sum_{k = t_1}^j \frac{1}{(jk)^{\frac{5}{4}}}+ \frac{k}{j^{\frac{5}{4}}} \exp(- \frac{2}{|\mathcal{E}|^2}k^{1/2}) +jk\exp(- \frac{2}{|\mathcal{E}|^2}j^{1/2}).
\]
It's now clear, for example from Lemma \ref{sumint}, that the lemma follows.
\end{proof}
\subsection{$m_t$: Martingale Convergence}
The goal of the subsection is to prove that $m_t$ converges. We begin by stating the $\mathcal{L}^2$ martingale convergence theorem without proof.
\begin{lemma}{\cite{r18}}{\label{mct}}
Let $y_t$ be a martingale with $y_t \in \mathcal{L}^2$ for all $t$. Further assume that $\sup_t \|y_t\|_{\mathcal{L}^2} < \infty$. Then $y_t$ converges in $\mathcal{L}^2$.
\end{lemma}

\begin{lemma}{\label{mconverges}}
$m_t$ converges in $\mathcal{L}^2$.
\end{lemma}
\begin{proof}
We first show that $m_t$ is a martingale. It is clearly an adpated process. Next, consider 
\[
\mathbb{E}_{t-1} [m_t - m_{t-1}] = \vec{p}\cdot \mathbb{E}_{t-1}\big[  \vec{\gamma}^T_{t}\circ L_{t-1}  -  \mathbb{E}_{t-1}[\vec{\gamma}^T_{t}\circ L_{t-1}]\big]\vec{x}_{t-1} +\vec{ p} \cdot \mathbb{E}_{t-1}[\vec{\gamma}^T_{t}\circ \vec{w}_{t}] .
\]
The first term is clearly $0$. We next note that:
\begin{align*}
\mathbb{E}_{t-1}[( \vec{\gamma}^T_{t}\circ \vec{w}_{t})^i] &=  \sum_{e \rightarrowtail i}   \mathbb{E}[ \frac{1}{g_t^i}\tilde{w}_{t}^e S_{t}^e | \mathcal{F}_{t-1}]\\
& =  \sum_{e \rightarrowtail i} \mathbb{E}[ \frac{1}{g_t^i} S_{t}^e \mathbb{E} [\tilde{W}_{t}^e | \sigma(\mathcal{F}_{t-1}, \sigma(\psi_t))]| \mathcal{F}_{t-1}]\\
&=  0.
\end{align*}

Lastly, note that
\[
m_t = a_t-a_0-s_t,
\] 
so that
\[
\|m_t\|_{\mathcal{L}^2} \leq \|a_t\|_{\mathcal{L}^2}+\|a_0\|_{\mathcal{L}^2}+\|s_t\|_{\mathcal{L}^2}.
\]
$a_0$ is a constant, $a_t$ is a.s. bounded by virtue of $0 \leq x_t \leq 1,$ and $\|s_t\|_{\mathcal{L}^2}$ is bounded since $s_t$ converges in $\mathcal{L}^2.$ Thus $m_t$ is bounded in $\mathcal{L}^2$, proving the theorem.
\end{proof}

\section{Decay of Disagreement}
The goal of this section is to show that the component of $\vec{x}_t$ corresponding to any differing opinions converges to $0$. Let $\vec{z}_t := \vec{x}_t - a_t \vec{1}$ represent this component of the opinion vector. We would like to show that

\begin{lemma}{\label{decaymain}}
\[
\mathbb{E}[\|\vec{z}_t\|^2] \rightarrow 0
\]
\end{lemma}

\subsection{Preliminary Discussion}

We develop our approach to a proof as follows. With $Q := \text{diag}(0,1,...,1)$, it's clear that $\vec{z}_t = PQP^{-1} \vec{x}_t$, where $P$ is the matrix of eigenvectors of $L$. Further, using the sum-product solution of the SHE from Proposition \ref{SHE}, we have that
\begin{align*}
\vec{z}_t &=  \sum_{j=0}^t PQP^{-1} [\Pi_{k=j}^{t-1} \Lambda_k ] (\vec{\gamma}_j^T \circ \vec{w}_j).
\end{align*}

The intuition for why $\vec{z}_t$ is small is as follows: at each past timestep $j \leq t$, a random 'blip' $\vec{\gamma}_j^T \circ \vec{w}_j$ was introduced. In subsequent time steps $k \geq j$, this blip was smoothed by repeated application of the $\Lambda_k$ matrices. Now, as argued in the previous section (see Lemma \ref{delta}), $\mathbb{E}[\Lambda_k] \approx A_k.$ Using $PQP^{-1}$ to project out the Perron-Frobenius eigenvalue $1$ of $A_k$ (corresponding to consensus), we get eigenvalues whose products decay sufficiently rapidly. So, sufficiently old blips are dampened by products of small eigenvalues with many factors, while newer blips will be small because the vector norm of $\vec{\gamma}_j$ is expected to decrease as $j$ increases. 

An issue with the above heuristic, however, is that random draws of $\vec{\gamma}_{k+1}^T \circ L_k$ are not close to $\frac{1}{k}L$ (even though they approximately agree in expectation). This is circumvented by noting that the $\vec{\gamma}^T_{k+1}\circ L_k$ are Ces\`aro-summable with limit proportional to $L$: we expand the product $\Pi_{k=j}^{t-1} \Lambda_k  = \Pi_{k=j}^{t-1} (I + \vec{\gamma}_{k+1}^T \circ L_k) $, show that the leading order terms (i.e. those linear in the dampened diffusion matrix) are proportional to $L$ due to a law of large numbers effect, and show that the lower order terms decay sufficiently rapidly because they have many factors of $\vec{\gamma}$. 

More precisely: we group the $t-j$ factors in the product $\Pi_{k=j}^{t-1} \Lambda_k$ into subgroups of size $\tau := \lceil t^{1/4} \rceil.$ This $\tau$ is large enough for the law of large numbers to kick in (allowing us to replace the group's average of the $\vec{\gamma}_{k+1}^T \circ L_k$ with a matrix proportional to $L$), but small enough so that there are enough factors of $L$ for the decay of the product of the non-dominant eigenvalues to be severe. Note that $j$ needs to be sufficiently small so that we have enough factors of $\Lambda$ to work with. With this in mind, we will separate the sum defining $z_t$ into $j \leq j_0$ and $j > j_0$ (for a value of $j_0$ to be specified later). The $j \leq j_0$ sum witnesses $PQP^{-1}  [\Pi_{k=j}^{t-1} \Lambda_k ]$ to have sufficiently small operator norm, while the $ j > j_0$ sum is small because we expect $\vec{\gamma}_j$ to be small at such late values of $j$. This heuristic is illustrated in Figure 2.

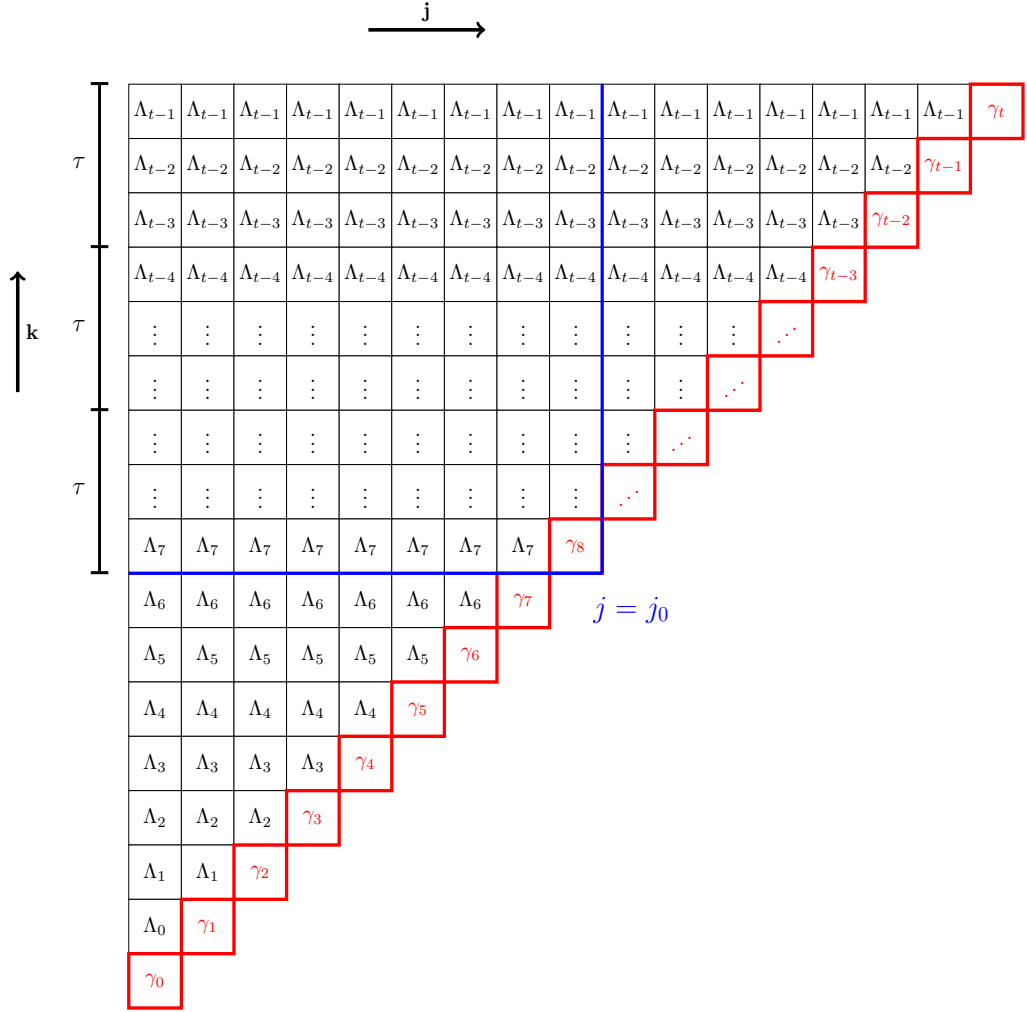
\begin{figure}
\resizebox{13.5cm}{13.5cm}{%
\begin{tikzpicture} 
\begin{centering}
\begin{scope}
	\def\max{17}
	\def\size{0.9}
      \foreach \p [count = \r] in {1,...,\max}{
	\tikzmath{\temp =int( \max - \p + 1);}
        \foreach \i in {1,...,\temp}{
	\tikzmath{\iret =int(\i-1);}
	\tikzmath{\pret =int(\p-1);}
	\tikzmath{\tempdoubleret =int(\temp-2);}
	\tikzmath{\tempret =int(\temp-1);}

	\ifthenelse{\p = 1}{\ifthenelse{\i < \max}{\draw[black,fill=white] (\size*\i, -\size*\r) rectangle  (\size*\i+\size, -\size-\size*\r)node[pos=.5, ] { $\Lambda_{t-1}$};}{\draw[red,fill=white, ultra thick] (\size*\i, -\size*\r) rectangle  (\size*\i+\size, -\size-\size*\r)node[pos=.5, ] { $\gamma_{t}$};}}{
		\ifthenelse{\p < 5}{\ifthenelse{\i < \temp}{\draw[black,fill=white] (\size*\i, -\size*\r) rectangle  (\size*\i+\size, -\size-\size*\r)node[pos=.5, ] { $\Lambda_{t-\p}$};}{\draw[red,fill=white, ultra thick] (\size*\i, -\size*\r) rectangle  (\size*\i+\size, -\size-\size*\r)node[pos=.5, ] { $\gamma_{t-\pret}$};}}{
		\ifthenelse{\p < 9}{\ifthenelse{\i < \temp}{\draw[black,fill=white] (\size*\i, -\size*\r) rectangle  (\size*\i+\size, -\size-\size*\r)node[pos=.5, ] {  $\vdots$};}{\draw[red,fill=white, ultra thick] (\size*\i, -\size*\r) rectangle  (\size*\i+\size, -\size-\size*\r)node[pos=.5, ] {  $\iddots$};}}{
			\ifthenelse{\i < \temp}{\draw[black,fill=white] (\size*\i, -\size*\r) rectangle  (\size*\i+\size, -\size-\size*\r)node[pos=.5, ] { $ \Lambda_{\tempdoubleret}$};}{\draw[red,fill=white, ultra thick] (\size*\i, -\size*\r) rectangle  (\size*\i+\size, -\size-\size*\r)node[pos=.5, ] {  $\gamma_{\tempret}$};}
			}
		}
	}

        }
      }
	
\draw  [line width=0.5mm, black ] (.4,-0.89) edge (.4,-9) node{} [right];

	\draw  [line width=0.5mm, black ] (.25,-.89) edge (.55,-.89) node{} [right];
	\draw  [line width=0.5mm, black ] (.25,-3.59) edge (.55,-3.59) node{} [right];
	\draw  [line width=0.5mm, black ] (.25,-6.29) edge (.55,-6.29) node{} [right];
	\draw  [line width=0.5mm, black ] (.25,-8.99) edge (.55,-8.99) node{} [right];

	\draw (0.05,-2.2) node{\large $\tau$};
	\draw (0.05,-4.9) node{\large $\tau$};
	\draw (0.05,-7.6) node{\large $\tau$};
	 \draw[->,line width = 0.6mm] (5,0) to ["$\textbf{j}$"] (7,0) ;
	\draw[<-,line width = 0.6mm] (-1,-4) to ["$\textbf{k}$"] (-1,-6) ;
	\draw [line width=.55mm, blue ] (9,-.9) -- (9,-9) node{} [right];
	\draw [line width=.55mm, blue ] (0.9,-9) -- (9,-9) node{} [right];
	\draw[blue] (9.5,-9.6) node{ \Large $j=j_0$};

  \end{scope}
\end{centering}
\end{tikzpicture}
} 
\caption{Schematic for the solution to the SHE: $\vec{x}_t = \sum_{j=0}^t [\Pi_{k=j}^{t-1} \Lambda_k ] (\vec{\gamma}_j^T \circ \vec{w}_j)$. $k$ parametrizes the factors in the product, $j$ parametrizes the terms in the sum. See the discussion above. We write $\gamma$ instead of $\gamma \circ w$ to save space.}
\end{figure}

For fixed $t$, define $r$ to be the remainder of $t$ divided by $\tau$, define $j_0 := r + \tau^2$, and let $H_{k,t}$ represent the aggregate effects of the $\Lambda$ factors from the $\tau$-window indexed by $k$. That is, for $1 \leq k \leq \frac{t-r}{\tau}$:
\[
H_{k,t} = [\Pi_{j=r+(k-1)\tau+1}^{r+k\tau} \Lambda_j ]
\]
so that, for sufficiently large $t$ and $j \leq j_0$,
\[
\Pi_{k = j}^t \Lambda_k = \big[\Pi_{k = \tau + 1}^{\frac{t-r}{\tau}} H_{k,t} \big] \big[\Pi_{k = j}^{j_0} \Lambda_k \big]
\]

\subsection{Good and Bad Events}
The above intuition only holds on 'good events' where the long-term randomness of the $\psi$ variables is close to expectation. In particular, we use this assumption when we assume $\vec{\gamma}_j$ to be small for large $j$, and that the Ces\`aro mean of $\vec{\gamma}_{k+1}^T \circ L_k$ is roughly proportional to $L$. For the rest of the paper we fix $\epsilon \ll \frac{1}{2}$, and for $\tau + 1 \leq k \leq \frac{t-r}{\tau}$, we define these good events as follows:
\begin{align*}
A_{k,t} &= \Bigl\{ |g_s^i - g_0^i - \frac{d_i}{E}k \tau | \leq (k \tau)^{\frac{1}{2} + \epsilon} \, : \, \forall i \in \mathcal{V}, \, \forall s \in \{(k-1)\tau + r + 1,..., k \tau+r\}\Bigr\} \\
B_{k,t} &= \{|(\sum_{s = (k-1) \tau+r+1}^{k \tau+r} S^e_s ) - \frac{1}{E} \tau| \leq \tau^{\frac{1}{2} + \epsilon}  \, : \, \forall e \in \mathcal{E}\}\\
E_t &= \cap_{k = \tau + 1}^{\frac{t-r}{\tau}} (A_{k,t} \cap B_{k,t}).
\end{align*}
$A_{k,t}$ corresponds to the event that, for all $s$ in the $\tau$-window indexed by $k$, $g_s - g_0$ is close to the expectation of $g$ at the point $k \tau$ (which lies in the window). The event $B_{k,t}$ represents that, within the $\tau$ window indexed by $k$, the amount of conversations each edge hosts is close to its expectation. $E_t$ is the intersection of the $A$ and $B$ events for all windows $\tau + 1 \leq k \leq \frac{t-r}{\tau}$.

We first establish that the union of the bad events have exponentially small probability in $t$.
\begin{lemma}{\label{badevents}}
For $\epsilon < \frac{1}{2},$ there exist positive constants $c_1$ and $c_2$ such that, for sufficiently large $t$,
\[
\mathbb{P}(E_t^c) \leq c_1 \exp(-c_2 \tau^{2 \epsilon}).
\]
\end{lemma}
\begin{proof}
\begin{align*}
\mathbb{P}(E_t^c) &= \mathbb{P}\big((\cap_{k = k_\text{min}}^{k_\text{max}}(A_{k,t} \cap B_{k,t}))^c\big) \\
&= \mathbb{P}\big((\cup_{k = k_\text{min}}^{k_\text{max}}(A^c_{k,t} \cup B^c_{k,t}))\big) \\
&\leq \sum_{k=k_\text{min}}^{k_\text{max}} \big( \mathbb{P}(A^c_{k,t})+\mathbb{P}(B^c_{k,t})\big),
\end{align*}
where we've invoked a union bound in the last line.

Now, for all $t\geq 0, k \geq 1$:
\[
\mathbb{P}(B_{k,t}^c) \leq 2|\mathcal{E}| \exp(-2 \tau^{2 \epsilon}).
\]
This follows directly from Lemma \ref{AS}, with union bound.

Similarly, for $t$ sufficiently large and $k \geq \tau$:
\[
\mathbb{P}(A^c_{k,t}) \leq 4 |\mathcal{V}| \exp(-\frac{1}{4} (k \tau)^{2 \epsilon}) \leq 4 |\mathcal{V}| \exp(-\frac{1}{4} \tau^{4 \epsilon}).
\]
The proof of this claim is as follows:
Note that 
\[
A_{k,t} = \{ g_{s_0}^i - g_0^i - \frac{d_i}{\mathcal{E}}k \tau  \geq -(k \tau)^{\frac{1}{2} + \epsilon} \, : \, \forall i \in \mathcal{V}\}\cap\{ g_{s_1}^i - g_0^i - \frac{d_i}{\mathcal{E}}k \tau  \leq (k \tau)^{\frac{1}{2} + \epsilon} \, : \, \forall i \in \mathcal{V}\}
\]
where $s_0$ and $s_1$ represent the endpoints for a particular $\tau$-window: \newline $s_0 = (k-1)\tau + r + 1$ and $s_1 = k\tau + r.$ Now, we have that
\begin{align*}
&\mathbb{P}\big(\{ g_{s_0}^i - g_0^i - \frac{d_i}{\mathcal{E}}k \tau  \geq -(k \tau)^{\frac{1}{2} + \epsilon} \, : \, \forall i \in \mathcal{V}\}^c\big) \\
&\leq |\mathcal{V}|\mathbb{P}\big(\{ g_{s_0}^i - g_0^i - \frac{d_i}{\mathcal{E}}k \tau  \leq -(k \tau)^{\frac{1}{2} + \epsilon}\}\big)\\
&=|\mathcal{V}|  \mathbb{P}\big(\{ g_{s_0}^i - g_0^i - \frac{d_i}{\mathcal{E}}s_0  \leq -(k \tau)^{\frac{1}{2} + \epsilon} + \frac{d_i}{\mathcal{E}}(\tau-r-1)\}\big) \Big)\\
&\leq|\mathcal{V}|  \mathbb{P}\big(\{ g_{s_0}^i - g_0^i - \frac{d_i}{\mathcal{E}}s_0  \leq -\frac{1}{2}(k \tau)^{\frac{1}{2} + \epsilon} \}\Big)\\
&\leq 2 |\mathcal{V}| \exp(-\frac{1}{4} (k \tau)^{2 \epsilon})
\end{align*}
Similarly, for the second event,
\begin{align*}
&\mathbb{P}\big(\{ g_{s_1}^i - g_0^i - \frac{d_i}{E}k \tau  \leq (k \tau)^{\frac{1}{2} + \epsilon} \, : \, \forall i \in \mathcal{V}\}^c\big) \\
&\leq | \mathcal{V}|\mathbb{P}\big(\{ g_{s_1}^i - g_0^i - \frac{d_i}{ \mathcal{\mathcal{E}}}k \tau  \geq (k \tau)^{\frac{1}{2} + \epsilon}\}\big)\\
&=|\mathcal{V}|  \mathbb{P}\big(\{ g_{s_1}^i - g_0^i - \frac{d_i}{\mathcal{E}}s_1  \geq (k \tau)^{\frac{1}{2} + \epsilon} - \frac{d_i}{\mathcal{E}}r\}\big) \Big)\\
&\leq |\mathcal{V}|  \mathbb{P}\big(\{ g_{s_1}^i - g_0^i - \frac{d_i}{\mathcal{E}}s_1  \geq\frac{1}{2} (k \tau)^{\frac{1}{2} + \epsilon}\}\big) \Big)\\
&\leq 2 |\mathcal{V}| \exp(-\frac{1}{4} (k \tau)^{2 \epsilon}).
\end{align*}
This concludes the proof of the above claim. We now finish by noting that
\[
\big( \mathbb{P}(A^c_{k,t})+\mathbb{P}(B^c_{k,t})\big) \leq 4(|\mathcal{V}|+|\mathcal{E}|)\exp(-\frac{1}{4} \tau^{2 \epsilon}),
\]
so that, for sufficiently large $t$:
\begin{align*}
\mathbb{P}(E_t^c) &\leq \sum_{k=\tau+1}^{\frac{t-r}{\tau}} \big( \mathbb{P}(A^c_{k,t})+\mathbb{P}(B^c_{k,t})\big) \\
&\leq  4(|\mathcal{V}|+|\mathcal{E}|)\frac{t}{\tau} \exp(-\frac{1}{4} \tau^{2 \epsilon})\\
&\leq  4(|\mathcal{V}|+|\mathcal{E}|) \exp(-\frac{1}{8} \tau^{2 \epsilon})
\end{align*}
\end{proof}
\subsection{Law of Large Numbers for Iterated Diffusion}
We now show that, on good events, $H_{k,t}$ (representing the time-evolution over the $\tau$-window indexed by $k$) window is close to $A_k$. We begin by analyzing the leading-order terms in the product. The following lemma shows that, on good events, the dampened Laplacian matrices are Ces\`aro-summable, with average close to the influence matrix.
\begin{lemma}{\label{cesaro}}
Fix $\epsilon \ll \frac{1}{2}$. There exists a constant $c$ such that, for $t$ sufficiently large, $k \geq \tau$ and on $A_{k,t} \cap B_{k,t}$:
\[
\|\big(\sum_{j = (k-1)\tau+z+1}^{k \tau+z} \vec{\gamma}_{j+1}^T \circ L_j\big) - \frac{1}{k} L \| \leq \frac{c}{k \tau^{1/2-\epsilon}}
\]
\end{lemma}
\begin{proof}
Fix vertices $i \sim \ell$, and consider outcomes on $A_{k,t} \cap B_{k,t}$ only. In the below, the constant $c$ may change from line to line, but will never depend on $t$ or $k$.
\begin{align*}
\big(&\sum_{j = (k-1)\tau+z+1}^{k \tau+z} \gamma_{j+1}^T \circ L_j\big)^{i\ell} = \sum_{j = (k-1)\tau+r+1}^{k \tau+r} \frac{1}{g_{j+1}^i} S_{j+1}^{i \ell} \frac{g_j^\ell}{g_j^i + g_j^\ell}\\
& \leq  \frac{1}{\frac{d_i}{|\mathcal{E}|}k \tau - (k\tau)^{\frac{1}{2} + \epsilon}+g_0^i} \frac{\frac{d_\ell}{|\mathcal{E}|}k \tau + (k\tau)^{\frac{1}{2} + \epsilon}+g_0^\ell}{\frac{d_\ell+d_i}{|\mathcal{E}|}k \tau - 2(k\tau)^{\frac{1}{2} + \epsilon}+g_0^i+g_0^\ell } \sum_{j = (k-1)\tau+r+1}^{k \tau+r} S_{j+1}^{i \ell}\\
& \leq  \frac{1}{\frac{d_i}{|\mathcal{E}|}k \tau - (k\tau)^{\frac{1}{2} + \epsilon}+g_0^i} \frac{\frac{d_\ell}{|\mathcal{E}|}k \tau + (k\tau)^{\frac{1}{2} + \epsilon}+g_0^\ell}{\frac{d_\ell+d_i}{|\mathcal{E}|}k \tau - 2(k\tau)^{\frac{1}{2} + \epsilon}+g_0^i+g_0^\ell } \big(\frac{\tau}{|\mathcal{E}|} + \tau^{\frac{1}{2}+\epsilon} +1\big)\\
& \leq \frac{|\mathcal{E}|}{k \tau} (L^{i \ell} + \frac{c}{(k \tau)^{\frac{1}{2}-\epsilon}})\big(\frac{\tau}{|\mathcal{E}|} + \tau^{\frac{1}{2}+\epsilon} + 1 \big)\\
&= (L^{i \ell} + \frac{c}{(k \tau)^{\frac{1}{2}-\epsilon}})\big(\frac{1}{k} + \frac{c}{k\tau^{\frac{1}{2}-\epsilon}} \big)\\
&\leq \frac{1}{k}L^{i \ell} + \frac{c}{k \tau^{\frac{1}{2}-\epsilon}}.
\end{align*}
The opposite-direction inequality can be proven similarly, giving that 
\[
|\big(\sum_{j = (k-1)\tau+r+1}^{k \tau+r} \vec{\gamma}_{j+1}^T \circ L_j\big)^{i\ell}  - \frac{1}{k}L^{i \ell}| \leq \frac{c}{k \tau^{\frac{1}{2}-\epsilon}}.
\]
From this the lemma easily follows.
\end{proof}

We now use the above lemma to show that the $H$ product matrix over the $k^\text{th}$ $\tau$-window is indeed close to $A_k$ (sub-leading order terms included). Define the difference $\Theta_{k,t} := H_{k,t} - A_k$
\begin{lemma}{\label{thetasmall}}
Fix $\epsilon \ll \frac{1}{2}$. There exists a constant $c$ such that, for sufficiently large $t$, $k \geq \tau$ and on $A_{k,t} \cap B_{k,t}$:
\[
\|\Theta_{k,t}\| \leq \frac{c}{k\tau^{1/2-\epsilon}}
\]
\end{lemma}
\begin{proof}
The matrix $H_{k,t} = \Pi_{j=(k-1)\tau + r + 1}^{k \tau + r} (I + \vec{\gamma}^T_{j+1} \circ L_j)$ has $\tau$ factors in the product. It can be expanded as a sum: 
\[
H_{k,t} = \sum_{n=0}^\tau h_{k,t,n}
\]
where $h_{k,t,n}$ collects the terms in the expansion with exactly $n$ factors of the dampened laplacian matrix $\vec{\gamma}^T \circ L$. Now, on good events, and for all $s_0 \leq j \leq s_1$ and for all $i \in \mathcal{V}$, we have that 
\[
g_s^i \gtrsim k \tau.
\]
It's also to see that, almost surely, we have that $\|L_j\| \lesssim 1$. Then, using the submultiplicativity of the operator norm with respect to the Hadamard product, we have that, on good events,
\[
\|\vec{\gamma}_{j+1}^T \circ L_j\| \leq \frac{c}{k \tau}
\]
for some constant $c$. So, collecting all terms with $n$ such matrices as factors in the binomial expansion (there are $\binom{\tau}{n}$ of them), we have that
\[
\|h_{k,t,n}\| \leq \binom{\tau}{n}  \big(\frac{c}{k \tau}\big)^n \leq \frac{1}{n!} \big(\frac{c}{k}\big)^n
\]
So, using the previous LEMMA (which says that $\|h_{k,t,0}+h_{k,t,1} - A_k\| \leq \frac{c'}{k \tau^{1/2-\epsilon}}$ for some constant $c'$:
\begin{align*}
\|H_{k,t}- A_k\| &\leq \frac{c'}{k \tau^{1/2-\epsilon}} + \sum_{n=2}^{\tau}  \frac{1}{n!} \big(\frac{c}{k}\big)^n\\
& \leq \frac{c'}{k \tau^{1/2-\epsilon}} + \big(\frac{c}{k}\big)^2\sum_{n=0}^{\tau} \big(\frac{c}{k}\big)^n\\
& \leq  \frac{c'}{k \tau^{1/2-\epsilon}} + \frac{(\frac{c}{k})^2}{1-\frac{c}{k}}\\
& \leq \frac{c'}{k \tau^{1/2-\epsilon}} + \frac{2c^2}{k \tau^{1/2-\epsilon}}\\
&= \frac{c' + 2c^2}{k \tau^{1/2-\epsilon}},
\end{align*}
where we've used that  $\tau$ is large, for example, enough to have $\frac{c}{\tau} \leq \frac{1}{2}$, and that $k \geq \tau.$
\end{proof}

\subsection{Decay of Operator Norm}
Now, recall that 
\[
\Pi_{k = j}^t \Lambda_k = \big[\Pi_{k = \tau + 1}^{\frac{t-r}{\tau}} H_{k,t} \big] \big[\Pi_{k = j}^{ j_0} \Lambda_k \big].
\]
The late ($k > j_0$) $\Lambda_k$ matrices in the product are encapsulated in the $H$ matrices, while we 'chop off' the early ($k \leq j_0$) $\Lambda_k$ matrices. We remove them because $t-j$ may not be divisible by $\tau$ (and thus that we cannot successfully partition all $\Lambda$ into groups of equal size). Note, however, that in the above decomposition, we chop off more than the remainder of $t-j$ divided by $\tau$; this is for later convenience. 

The next Lemma guarantees that these extra, 'loose' factors of $\Lambda$ have bounded norm. We present a straightforward proof which makes use of some simple matrix calculations. It can be noted, however, that this lemma can also be proven by noting that a discrete dynamical system driven by the $\Lambda$ matrices (with no random blips $W$) represent a version of the heat equation where the only randomness is in the edge selection, rather than in the outcome in the conversation, and long-term solutions must be bounded.
\begin{lemma}{\label{detheat}}
For all $0 \leq j \leq t$,
\[
\|\Pi_{k = j}^{t} \Lambda_k \|\leq \sqrt{|\mathcal{V}|}
\]
almost surely.
\end{lemma}
\begin{proof}
We first aim to prove that $\Lambda_k$ is nonnegative. The offdiagonal elements are obviously nonnegative, so we focus only on the diagonal. Let $k$ be arbitrary. For any $i$,
\[
\Lambda_k^{ii} = 1 -\frac{1}{g_{k+1}^i}\sum_{j \sim i}S_{k+1}^{ij}\frac{g_k^j}{g_k^i + g_k^j}.
\]
Now, if $S_{k+1}^{ij} = 0$ for all $j \sim i,$ then it's clear that $\Lambda_{k}^{ii} = 1$. Otherwise, let $S_{k+1}^{ij} = 1$ for some $j \sim i.$ Note that this will be the only nonzero term in the sum. In this case, we are guaranteed that $g_{k+1}^i = g_{k}^i + 1 > 1.$ So: 
\[
\Lambda_k^{ii} = 1 -\frac{1}{g_{k+1}^i}\frac{g_k^j}{g_k^i + g_k^j} >  1 -\frac{g_k^j}{g_k^i + g_k^j} > 0.
\]
This gives that $\Lambda_k$ is nonnegative.

Fix $j,t$ as above, arbitrary. Note that 
\[
A_{j,t} := \Pi_{k = j}^{t} \Lambda_k
\]
is row stochastic, as it is the product of row stochastic matrices. It's also nonnegative. So, let $\vec{v}$ be an aribtrary unit vector. Note that for all $j$, $|v^j| \leq 1$ so:
\[
|(A_{j,t}v)_i| = |\sum_k (A_{j,t})^{ik} v^k| \leq  \sum_j (A_{j,t})^{ik} = 1.
\]
where we've used that $A$ is nonnegative. Thus for arbitrary unit vector, $\|A_{j,t}v\|^2 = \sum_k \big((A_{j,t}v)^k\big)^2 \leq \sum_k  1= |\mathcal{V}|$.
This proves that, for arbitrary $0 \leq j < t$, almost surely, $\|A_{j,t}\| \leq \sqrt{|\mathcal{V}|}$.
\end{proof}

Before tackling the main lemma of this section (Lemma \ref{opnorm}), we note the useful fact that for a square matrix $A$, the $\ell_2$ operator norm is equivalent to the max of the vector norms of the rows.

\begin{lemma}{\label{matrixnorm}}
\label{lemma:rownorm}
Let $A$ be an $n\times$n matrix, and let $A^i$ represent the ith row of $A$. We then have the following two inequalities, for arbitrary $1 \leq i \leq n$:
\begin{align*}
\|A^i\| &\leq \|A\|\\
\|A\| &\leq \sqrt{n}\max_{j} \|A^j\| \\
\end{align*}
Note that for $1\times n$ matrices, the matrix norm coincides with the vector norm.
\end{lemma}
\begin{proof}
Let $1\leq i \leq n$ be arbitrary. We prove the first inequality. If $A^i = \vec{0},$ we are done. Otherwise, define the vector $\vec{x}$ to have components $x^j := \frac{A^{ij}}{\|A^i\|}$. Then we have
\[
\|A\| \geq \|A\vec{x}\| \geq \|A^i\|,
\]
where the second inequality follows because the $i^\text{th}$ entry of $A \vec{x}$ is equal to $A^i \cdot x = \|A^i\|$.

Next we prove the second inequality. Let $\vec{x}$ be an arbitrary unit vector. We have
\[
\|A\vec{x}\|^2 = \sum_{j} (A^j \cdot \vec{x})^2 \leq  \sum_{j} \|A^j\|^2 \leq n \max_j \|A^j\|^2.
\]
Taking the square root of both sides, we have the desired inequality.
\end{proof}

We now add the main ingredient in the proof of Lemma 5.1, which says that for sufficiently small $j$ and on good events, the product of diffusion matrices (with consensus projected out) decays with $t$.

\begin{lemma}{\label{opnorm}}
There exists $\alpha > 0$ such that, on $E_t$, and for all $j \leq j_0 := r + \tau^2,$
\[
\|Q P^{-1}[\Pi_{k=j}^{t} \Lambda_k] P\| \lesssim \frac{1}{t^\alpha},
\]
\end{lemma}
\begin{proof}
Define $k_\text{min} := \tau + 1$, $k_{\text{max}} = \frac{t-r}{\tau}$, $\Theta'_{k,t} := P^{-1} \Theta_{k,t} P$, and $D_k = P^{-1}A_k P$ (the diagonal matrix consisting of eigenvalues of $A_k$). Now,
\[
\big[\Pi_{k = k_\text{min}}^{k_\text{max}} H_{k,t} \big] \big[\Pi_{k = j}^{ j_0} \Lambda_k \big].
\]

\begin{align*}
\|Q P^{-1}[\Pi_{k=j}^{t} \Lambda_k] P\| &\leq \|Q P^{-1} [\Pi_{k=j_0+1}^{t}\Lambda_k]P \| \|P^{-1} [\Pi_{k=j}^{j_0}\Lambda_k]P\|\\
&\lesssim \|Q P^{-1} [\Pi_{k=j_0+1}^{t}\Lambda_k]P \|\\
&= \|Q P^{-1} [\Pi_{k=k_{\text{min}}}^{k_\text{max}}H_{k,t}]P \|\\
&= \|Q P^{-1} [\Pi_{k=k_{\text{min}}}^{k_\text{max}}(A_k + \Theta_{k,t})]P \|\\
&= \|Q  [\Pi_{k=k_{\text{min}}}^{k_\text{max}}(D_k + \Theta'_{k,t})] \|\\
&\lesssim \max_{i \neq 1} \|  [\Pi_{k=k_{\text{min}}}^{k_\text{max}}(D_k + \Theta'_{k,t})]^i \|\\
&= \max_{i \neq 1} \|R^i_{k_\text{max},k_\text{min},t}\|
\end{align*}
where we've used Lemma \ref{matrixnorm} in the second to last line and we've defined 
\[
R_{k_1,k_2,t} := \Pi_{k=k_1}^{k_2}(D_k + \Theta'_{k,t}).
\]
We have, for a constant $c$, for $\epsilon < \frac{1}{2},$ and for all $i \neq 1$ (dropping primes on Theta for ease of notation),
\begin{align*}
R_{k_\text{min},k_\text{max},t}^i &= \Theta_{k_\text{max},t}^i R_{k_\text{min},k_\text{max}-1,t} + |\lambda^{(k_\text{max})}|R_{k_\text{min},k_\text{max}-1,t}^i \\
\|R_{k_\text{min},k_\text{max},t}^i\|&\leq \frac{c}{k_\text{max} \tau^{1/2-\epsilon}} + (1 - \frac{\lambda}{k_\text{max}})\|R_{k_\text{min},k_\text{max}-1,t}^i\|,
\end{align*}
where $\lambda^{(k)} \neq 1$ is an eigenvalue of $A_k$, and we've used Lemma \ref{detheat}, Lemma \ref{thetasmall}, and Lemma \ref{pfspec}. By iterating the above, we obtain
\begin{align*}
\|R_{k_\text{min},k_\text{max},t}^i\|&\leq \Pi_{k=k_\text{min}}^{k_\text{max}} (1-\frac{\lambda}{k}) +\frac{c}{\tau^{1/2-\epsilon}} \sum_{j=k_\text{min}}^{k_\text{max}} \frac{1}{j} \Pi_{k = j+1}^{k_\text{max}}  (1-\frac{\lambda}{k})\\
&\leq \big(\frac{k_\text{min}}{k_\text{max}}\big)^\lambda + \frac{c}{\tau^{1/2-\epsilon}}  \sum_{j=k_\text{min}}^{k_\text{max}} \frac{1}{j} \big(\frac{j+1}{k_\text{max}}\big)^\lambda \\
&\leq \big(\frac{k_\text{min}}{k_\text{max}}\big)^\lambda + \frac{c}{k_\text{max}^\lambda \tau^{1/2-\epsilon}}  \sum_{j=k_\text{min}}^{k_\text{max}} j^{\lambda-1} \\
&\leq \big(\frac{k_\text{min}}{k_\text{max}}\big)^\lambda+ \frac{c}{\tau^{1/2-\epsilon}} \\
&\lesssim \frac{1}{t^{\lambda/2}} + \frac{1}{t^{1/8-\epsilon/4}},
\end{align*}
where in the second and fourth inequalities, we used Lemma \ref{gautschiapp} and Lemma \ref{sumint}, respectively, and the value of $c$ can change from line to line. Setting $\epsilon = \frac{1}{4}$ (for example) concludes the proof of the lemma.
\end{proof}

Our final ingredient is the summability of $\mathbb{E}[\|\vec{\gamma}_t\|^2].$
\begin{lemma}{\label{gammasum}}
The sum
\[
\sum_{t=0}^\infty \mathbb{E}[\|\vec{\gamma}_t\|^2]
\]
converges.
\end{lemma}
\begin{proof}
For sufficiently large $t$:
\begin{align*}
\mathbb{E}[\|\vec{\gamma}_t\|^2] &= \mathbb{E}[\|\vec{\gamma}_t\|^2 |A_{t,\frac{t-r}{\tau}} ]\mathbb{P}(A_{t,\frac{t-r}{\tau}}) + \mathbb{E}[\|\vec{\gamma}_t\|^2 |A_{t\frac{t-r}{\tau}}^c ]\mathbb{P}(A_{t,\frac{t-r}{\tau}}^c) \\
&\lesssim \mathbb{E}[\|\vec{\gamma}_t\|^2 |A_{t,\frac{t-r}{\tau}} ]+\mathbb{P}(A_{t,\frac{t-r}{\tau}}^c)\\
&\lesssim \frac{1}{t^2} +  \exp(-\frac{1}{4} \tau^{4 \epsilon})\\
&\lesssim \frac{1}{t^2}
\end{align*}
\end{proof}

\subsection{Proof of Lemma 4.1}

In the proof of Lemma 4.1, we make use of the following simple comparison between a nonnegative random variable's conditional and total expectation:
\begin{lemma}{\label{conditionalexpectation}}
Let $X$ be an almost-surely nonnegative random variable with $\mathbb{E}[X] < \infty,$ and let $E$ be an event with $\mathbb{P}(E) \geq \frac{1}{2}.$ Then
\[
\mathbb{E}(X|E) \leq 2 \mathbb{E}[X]
\]
\end{lemma}
\begin{proof}
\[
\mathbb{E}[X|E] = \frac{1}{\mathbb{P}[E]}(\mathbb{E}[X] - \mathbb{E}[X|E^c]\mathbb{P}(E^c)) \leq 2 \mathbb{E}[X].
\]
\end{proof}

\begin{proof}[Proof of Lemma 4.1:]
We aim to show that
\[
\mathbb{E}[\|\vec{z}_{t+1}\|^2] \rightarrow 0,
\]
where
\[
\vec{z}_{t+1} = \sum_{j = 0}^{t+1} P Q P^{-1}\big[\Pi_{k=j}^{t} \Lambda_k \big] \vec{\gamma}^T_j \circ \vec{w}_j.
\]
Expanding the square:
\begin{align*}
\|\vec{z}_{t+1}\|^2 &= \sum_{j = 0}^{t+1} \|P Q P^{-1}\big[\Pi_{k=j}^{t} \Lambda_k \big]  \vec{\gamma}^T_j \circ \vec{w}_j\|^2 \\
&+ 2 \sum_{0 \leq j_1 < j_2 \leq t+1} \langle P Q P^{-1}\big[\Pi_{k_1=j_1}^{t} \Lambda_{k_1} \big]   \vec{\gamma}^T_{j_1} \circ \vec{w}_{j_1} , P Q P^{-1}\big[\Pi_{k_2=j_2}^{t} \Lambda_{j_2}\big]    \vec{\gamma}^T_{j_2} \circ \vec{w}_{j_2} \rangle
\end{align*}
We now take the expectation of the cross-terms. For $0 \leq j_1 < j_2 \leq t+1$:
\begin{align*}
&\mathbb{E}[ \langle P Q P^{-1}\big[\Pi_{k_1=j_1}^{t} \Lambda_{k_1} \big]   \vec{\gamma}^T_{j_1} \circ  \vec{w}_{j_1} , P Q P^{-1}\big[\Pi_{k_2=j_2}^{t} \Lambda_{j_2}\big]    \vec{\gamma}^T_{j_2} \circ \vec{w}_{j_2} \rangle]\\
&=\mathbb{E}\big[\mathbb{E}[\langle P Q P^{-1}\big[\Pi_{k_1=j_1}^{t} \Lambda_{k_1} \big]   \vec{\gamma}^T_{j_1} \circ \vec{w}_{j_1} , P Q P^{-1}\big[\Pi_{k_2=j_2}^{t} \Lambda_{j_2}\big]    \vec{\gamma}^T_{j_2} \circ \vec{w}_{j_2} \rangle|\sigma(\mathcal{F}_{j_2-1},\mathcal{G}_{t+1})] \big]\\
&=\mathbb{E}\big[\big( P Q P^{-1}\big[\Pi_{k_1=j_1}^{t} \Lambda_{k_1} \big]   \vec{\gamma}^T_{j_1} \circ \vec{w}_{j_1} \big)^TP Q P^{-1}\big[\Pi_{k_2=j_2}^{t} \Lambda_{j_2}\big] ( \vec{\gamma}^T_{j_2} \circ \mathbb{E}[\vec{w}_{j_2}|\sigma(\mathcal{F}_{j_2-1},\mathcal{G}_{t+1}) \big]\\
&= 0,
\end{align*}
where we've used independence of $\sigma(\psi_t)$ and $\sigma(\mathcal{F}_{t-1},\sigma(\Omega_t))$ as well as the fact that $\mathbb{E}_{t-1}[\vec{w}_t]=0$.

We now deal with the expectation of the 'diagonal' elements:
\begin{align*}
\mathbb{E}[\|\vec{z}_t\|^2] &=  \sum_{j = 0}^{t+1} \|P Q P^{-1}\big[\Pi_{k=j}^{t} \Lambda_k \big]  \vec{\gamma}^T_j \circ \vec{w}_j\|^2 \\
&\lesssim \sum_{j = 0}^{t+1}  \mathbb{E} [\|P Q P^{-1}\big[\Pi_{k=j}^{t} \Lambda_k \big]\|^2  \| \vec{\gamma}_j\|^2 ] \\
&=  \sum_{j=0}^{ r+\tau^2 } \mathbb{E} \big[ \|P Q P^{-1}\big[\Pi_{k=j}^{t} \Lambda_k \big]\|^2  \| \vec{\gamma}_j\|^2  \big] + \sum_{j = r+\tau^2 + 1}^{t+1} \mathbb{E} \big[\|P Q P^{-1}\big[\Pi_{k=j}^{t} \Lambda_k \big]\|^2  \| \vec{\gamma}_j\|^2  \big] 
\end{align*}
We show that each of the above two terms goes to zero. Let $\alpha > 0$ be the number guaranteed by Lemma \ref{opnorm}. The first term:
\begin{align*}
&\sum_{j = 0}^{r+\tau^2}\mathbb{E} \big[\|P Q P^{-1}\big[\Pi_{k=j}^{t} \Lambda_k \big]\|^2  \|\vec{\gamma}_j\|^2  \big]\\
&\lesssim \sum_{j = 0}^{r+\tau^2} \mathbb{E} \big[\|P Q P^{-1}\big[\Pi_{k=j}^{t} \Lambda_k \big]\|^2 \|\vec{\gamma}_j\|^2  |E_t\big] + \mathbb{E} \big[ \|P Q P^{-1}\big[\Pi_{k=j}^{t} \Lambda_k \big]\|^2 \|\vec{\gamma}_j\|^2  |E_t^c\big] \mathbb{P}(E_t^c)\\
&\lesssim  \sum_{j = 0}^{r+\tau^2} \frac{1}{t^{\alpha}} \mathbb{E} \big[ \|\vec{\gamma}_j\|^2  |E_t\big] +\mathbb{P}(E_t^c) \\
&\lesssim  \sum_{j = 0}^{r+\tau^2} \frac{1}{t^{\alpha}} \mathbb{E} \big[ \|\vec{\gamma}_j\|^2 \big] +\mathbb{P}(E_t^c) \\
&\lesssim   \frac{1}{t^{\alpha}}  + (r+\tau^2) \exp(-c \tau^{2\epsilon})\\
& \rightarrow 0,
\end{align*}
where in the second inequality we used Lemmas \ref{opnorm} and \ref{detheat}, in the third inequality we used Lemma \ref{conditionalexpectation}, and in the fourth inequality we used Lemmas \ref{gammasum} and \ref{badevents}.

And in the second term of the expansion of the diagonal sum:
\begin{align*}
&\sum_{j = r + \tau^2 + 1}^{t+1} \mathbb{E} \big[\|P Q P^{-1}\big[\Pi_{k=j}^{t} \Lambda_k \big]\|^2  \|\vec{\gamma}_j\|^2  \big]
 \lesssim \sum_{j = r + \tau^2 + 1}^{t+1} \mathbb{E} \big[  \|\vec{\gamma}_j\|^2  \big],
\end{align*}
where we've used Lemma \ref{detheat}. The right-hand side goes to $0$ By Lemma \ref{gammasum}, since the lower bound of the sum goes to $\infty.$

\end{proof}

\section{Proof of Theorem}
\begin{proof}[Proof of Theorem 1.1]
Let $a_\infty$ be the limit of $a_t = \vec{p} \cdot \vec{x}_t$, established in Lemma \ref{consensusmain}. Using the triangle inequality, we have:
\begin{align*}
\mathbb{E}[\|\vec{x}_t - a_\infty \vec{1}\|^2] &= \mathbb{E}[\|\vec{x}_t -a_t \vec{1} + a_t \vec{1} - a_\infty \vec{1}\|^2] \\
&\leq \mathbb{E}[\| a_t \vec{1} - a_\infty \vec{1}\|^2] +\mathbb{E}[\|\vec{x}_t -a_t \vec{1} \|^2] + 2\mathbb{E}[\|\vec{x}_t -a_t \vec{1}\| \| a_t \vec{1} - a_\infty \vec{1}\|] .
\end{align*}
We now show that each term goes to $0$. In the first term, we have that 
\[
\| a_t \vec{1} - a_\infty \vec{1}\|^2 = (a_t - a_\infty)^2 \|\vec{1}\|^2,
\]
the expectation of which goes to $0$ by virtue of Lemma \ref{consensusmain}. Similarly, the second term goes to $0$ due to Lemma \ref{decaymain}.

To see that the last term goes to $0$, note that $\|\vec{x}_t -a_t \vec{1}\|$ is almost surely bounded, so that
\[
\mathbb{E}[\|\vec{x}_t -a_t \vec{1}\| \| a_t \vec{1} - a_\infty \vec{1}\|]  \lesssim \mathbb{E}[\| a_t \vec{1} - a_\infty \vec{1}\|] = \|\vec{1}\| \mathbb{E}[| a_t - a_\infty|] \lesssim \| a_t - a_\infty\|_{\mathcal{L}^1}.
\]
Finally, since $ a_t \rightarrow a_\infty$ in $\mathcal{L}^2,$ convergence also holds in $\mathcal{L}^1,$ so that this term goes to $0$ as well.
\end{proof}
\newpage

\section{Future Work}
Future work might consider the rate of convergence, for example of the disagreement component $\vec{z}_t$ to $0$. Simulations inspire the following conjecture:
\begin{conjecture}
\[
\mathbb{E}[\|\vec{z}_t\|^2] \lesssim
\begin{cases}
\frac{1}{t^{2 \lambda}}  & \lambda \leq \frac{1}{2} \\
\frac{1}{t} & \lambda > \frac{1}{2}
\end{cases}
\].
\end{conjecture}

In the case of parallel updates (i.e. all edges converse with all of their neighbors simultaneously in each time step), the above conjecture can be proven readily using the techniques from Lemma \ref{opnorm}. With the appropriate choice of $\tau (t)$, bounds on the decay rate can be proven for the present case (though these bounds seem loser than what simulation demonstrates). This discussion has been omitted because the bounds do not seem empirically tight, and the choice of $\tau (t) = \lceil t^{1/4}\rceil$ is convenient.

\begin{figure}
 \includegraphics[width=4in]{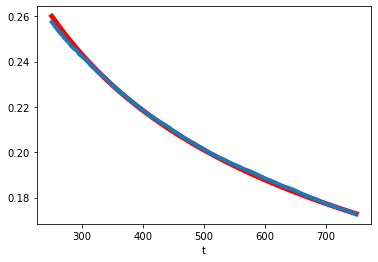}\hspace{.5em}%
\caption{The blue curve represents the average of $1000$ trajectories of $\|\vec{x}_t - a_t \vec{1} \|^2_2$, for $g^i_0 \equiv 1$ and $\vec{x}_0 = (1,1,0,0,0)$ for the graph $I_5$. \newline The red curve is $\frac{2.02}{t^{2*0.185667}}$ \newline (note: for $I_5$, $\lambda = \frac{13-\sqrt{73}}{24} \approx 0.185667$).}
\end{figure}

Figure 3 shows the decay of disagreement, averaged over $1000$ runs for the interval graph $I_5 = (\mathcal{V},\mathcal{E})$, where $\mathcal{V} = \{1,2,3,4,5\},$ and $(i,j) \in \mathcal{E}$ if and only if $|i-j| = 1$.

\section{Acknowledgment}
The author thanks Lionel Levine for his guidance throughout this research, particularly for his advice concerning the law of large numbers approach in Section 4.

\end{document}